\documentclass[11pt,a4paper]{article}

\usepackage{amsmath,amscd,amssymb,amsthm}
\usepackage{xspace}
\usepackage[dvips]{epsfig,graphics}
\usepackage{graphicx,psfrag}
\usepackage{tikz}
\usepackage{times}
\usepackage{color}
\usepackage[top=2.6cm, bottom=2.6cm, left=3cm, right=3cm]{geometry}
 \font\tenmsb=msbm10
 \font\sevenmsb=msbm7
 \font\fivemsb=msbm5
 \newfam\msbfam
 \textfont\msbfam=\tenmsb
 \scriptfont\msbfam=\sevenmsb
 \scriptscriptfont\msbfam=\fivemsb

\let\oldenumerate\enumerate
\renewcommand{\enumerate}{
  \oldenumerate
  \setlength{\itemsep}{1pt}
  \setlength{\parskip}{0pt}
  \setlength{\parsep}{0pt}
}

\let\oldenumerate\itemize
\renewcommand{\itemize}{
  \oldenumerate
  \setlength{\itemsep}{1pt}
  \setlength{\parskip}{0pt}
  \setlength{\parsep}{0pt}
}

\title{Towards a Classification of Countable 1-Transitive Trees: Countable Lower 1-Transitive Linear Orders}
\author{Silvia Barbina, Katie Chicot\footnote{The authors wish to thank Professor J.K.Truss and Professor H.D. MacPherson for their extensive help. This paper forms part of the second author's PhD thesis at the University of Leeds, which was supported by EPSRC grant EP/H00677X/1.} \\ The Open University }

\begin{document}

\maketitle
\newcommand{\R}{\ensuremath{\mathbb{R}}}
\newcommand{\Q}{\ensuremath{\mathbb{Q}}}
\newcommand{\Z}{\ensuremath{\mathbb{Z}}}
\newtheorem{lemma}{Lemma}[section]
\newtheorem{theorem}[lemma]{Theorem}
\newtheorem{corollary}[lemma]{Corollary}
\newtheorem{definition}[lemma]{Definition}
\newtheorem{remark}[lemma]{Remark}

\begin{abstract}
This paper contains a classification of countable lower 1-transitive linear orders. This is the first step in the classification of countable 1-transitive trees given in \cite{me}. 
The notion of lower 1-transitivity
generalises that of 1-transitivity for linear orders, and is essential for the structure theory of
 1-transitive
trees. The classification is given in terms of `coding trees'. These describe how a linear order is fabricated from simpler
 pieces using concatenations, lexicographic products and other kinds of construction. We define coding trees and show how they encode lower 1-transitive linear orders. Then we show that a coding tree can
 be recovered from a lower 1-transitive linear order $(X, \leqslant)$ by examining all the invariant partitions on $X$.
\end{abstract}

\section{Introduction}

This paper extends a body of classification results for countably infinite ordered structures,
 under various homogeneity assumptions. As background we mention that
 Morel  \cite{Morel} classified the countable 1-transitive linear orders, of which there are $\aleph_1$,
Campero-Arena and Truss \cite{TrussGabi} extended this classification to
{\em coloured} countable 1-transitive linear orders, and Droste \cite{Droste} classified the
countable 2-transitive trees. The work of Droste was later generalised by Droste, Holland and Macpherson
 \cite{Drost&Holl&Macph}
to give a classification of all countable `weakly 2-transitive' trees; there are $2^{\aleph_0}$ non-isomorphic such trees.
The goal of this paper, and of \cite{me}, is to extend this last classification result to a considerably richer class,
by working under a much weaker symmetry hypothesis, namely 1-transitivity.

We first define the terminology used above and later. A \textbf{tree} is  a partial order
 in which any two elements have a common lower bound and the lower bounds of any element are linearly ordered.
A relational structure is said to be \textbf{$k$-transitive} if for any two isomorphic $k$-element substructures there is an automorphism
taking the first to the second. For partial orders, there is a notion, called \textit{weak 2-transitivity}, that generalises that of 2-transitivity: a partial order is weakly 2-transitive if for any two 2-element chains there is an automorphism taking
 the first to the second (but not necessarily for 2-element antichains).

A weaker notion still
is that of \textit{1-transitivity}.
The classification  of countable 1-transitive trees is considerably more involved  than that of
the weakly-2-transitive trees, and it rests on the classification of countable lower 1-transitive linear orders --- the subject of this paper.

\begin{definition}
\noindent A linear order $(X , \leqslant)$ is \textbf{lower 1-transitive} if
$$(\forall a, b \in X)  \ \{ x \in X : x \leqslant a\} \cong \{x \in X : x \leqslant b\}.$$
\end{definition}

An example of a lower 1-transitive, not 1-transitive linear order
is $\omega^{\ast}$, (that is, $\omega$ reversed).
It is easy to see that any branch (that is, maximal chain) of a 1-transitive tree must be lower 1-transitive, though it is  not necessarily 1-transitive.

 The natural relation of equivalence between lower 1-transitive linear
orders is \textit{lower isomorphism}, rather than isomorphism.

\begin{definition}
Two  linear orders, $(X , \leqslant)$ and  $(Y , \leqslant)$ are  \textbf{lower isomorphic} if
$$(\forall a \in X)(\forall b \in Y) \  \{ x \in X : x \leqslant a\} \cong \{y \in Y : y \leqslant b\}.$$
When this happens, we write $(X , \leqslant) \cong_l (Y , \leqslant)$.

\end{definition}

We shall use interval notation from now on where appropriate, that is,  
$$(-\infty,a]:=\{ x \in X : x \leqslant a\}\, .$$
With this notation, the isomorphisms in the above definitions may then be written more succinctly as $(- \infty , a] \cong (- \infty , b]$.

 The classification of countable lower 1-transitive linear orders is rather involved and so the current paper
 is devoted entirely to this, and the resultant classification of countable 1-transitive trees is deferred to \cite{me}.
 A principal feature of the classification of coloured 1-transitive countable linear orders,
 given in  \cite{gabithesis} and \cite{TrussGabi},
is the use of \textit{coding trees} to describe the construction of the orderings. In these papers, coding trees play a totally different role from that of
the 1-transitive trees which we aim to classify: they are classifiers, rather than structures being classified. In order to emphasise this distinction, and to be consistent
with previous references such as \cite{Droste} and \cite{gabithesis}, we adopt the convention that coding trees
`grow downwards',
that is, any two elements have a common upper bound, and the upper bounds of any element are linearly ordered.

Section~2 of this paper contains the definition of coding tree and related notions. Section~3 describes how to recover a linear order from a coding tree. The main work of the paper is in Section~4, where we show how to construct a coding tree from a linear order.
The main theorem is Theorem~\ref{T4}, which, in conjunction  with Theorem~\ref{T3}, gives our classification.

In order to give the flavour of the classification, we conclude this introduction with some examples of lower 1-transitive  linear orders. First, some notation and terminology are needed.

Let $(A, \leqslant)$,$(B, \leqslant)$ be linear orders; for convenience, we often omit the order symbol. Then $A.B$ denotes the lexicographic product of $A$ and $B$, where for $(a,b), (a^\prime, b^\prime) \in A \times B$, 
$(a, b) \leqslant (a^\prime, b^\prime)$ if and only if $a < a^\prime$, or $a = a^\prime$ and $b \leqslant b^\prime$. Also, $A+B$ denotes $A$ followed by  $B$, that is, the disjoint union of $A$ and $B$ with $a<b$ for all $a\in A$ and $b\in B$. We write $\dot{\Q}$ for ${\Q}+\{+\infty\}$. If $A$ is a linear order, then $A^*$ denotes the ordering with the same domain and the reverse order. If $n \in \mathbb{N} \cup \{ \aleph_0 \}$, then ${\Q}_n$ is a countable dense linear order  coloured by $n$ colours $c_0,\ldots,c_{n-1}$ and such that between any two distinct points there is a point of each colour. Likewise, $\dot{\Q}_n$ is ${\Q}_n +\{+\infty\}$, where the point $+\infty$ is also coloured by any of the $c_i$, or indeed any other colour. If $Y_0,\ldots,Y_{n-1}$ are linear orders, ${\Q}_n(Y_0,\ldots,Y_{n-1})$  denotes the ordering
 obtained by replacing each point of ${\Q}_n$ coloured $c_i$ by a convex copy of $Y_i$ (with the natural induced ordering).
If $n=\aleph_0$, we write
 ${\Q}_{\aleph_0}(Y_0,Y_1,\ldots)$.

The simplest countable lower 1-transitive linear orders are singletons, then  ${\omega}^{\ast}$ and $\Z$ (which are lower isomorphic), and $\Q$  and $\dot{\Q} $ (which are also lower isomorphic). These orders are the basic building blocks for our constructions. We obtain new lower 1-transitive linear orders by concatenating and taking lexicographic products of existing ones. More precisely, Theorem~\ref{T3} implies that if $A$ and $B$ are any lower 1-transitive
linear orders which are lower isomorphic, then ${\omega}^{\ast}.A + B$ is lower 1-transitive. 

For example, a lower isomorphism class (that is, a class of lower-isomorphic linear orders) consists of ${\Z}.{\Z}$, which by convention we write as ${\Z}^{2}$, ${\omega}^{\ast}.{\Z} + {\Z}$ and ${\omega}^{\ast}.{\Z} + {\omega}^{\ast}$.
 Note that we can concatenate ${\omega}^{\ast}.{\Z}$ with either ${\Z}$ or ${\omega}^{\ast}$ and the resulting linear order will still be lower 1-transitive.
 This is because ${\omega}^{\ast}$ has a right-hand endpoint and because ${\Z}$ and ${\omega}^{\ast}$ are lower isomorphic.
 Notice that ${\omega}^{\ast}.A + A \cong {\omega}^{\ast}.A$. We use the former form to streamline subsequent definitions in the paper. A yet more complex lower isomorphism class is that of ${\Z}^{3}$, which includes
 ${\omega}^{\ast}.{\Z}^{2} + {\Z}^{2}$, $ {\omega}^{\ast}.{\Z}^{2} + {\omega}^{\ast}.{\Z} + {\Z}$ and
 $   {\omega}^{\ast}.{\Z}^{2} + {\omega}^{\ast}.{\Z} + {\omega}^{\ast}$.

Theorem~\ref{T3} gives another construction of lower 1-transitive linear orders from existing ones. This construction involves the building block $\Q$: the linear order ${\Q}_n(Y_0,\ldots,Y_{n-1})$ (possibly with $n=\aleph_0$) is lower 1-transitive provided the $Y_{i}$ are lower isomorphic to each other. Moreover, as above, ${\Q}_n(Y_0,\ldots,Y_{n-1})+ Y$ is lower 1-transitive provided $Y$ and the $Y_{i}$ are all lower isomorphic to each other. 
 A simple example is $X = {\Q}_{2}({\omega^{\ast}}, {\Z})$, which is
 countable and lower 1-transitive. Its lower isomorphism class also includes
${\Q}_{2}( {\omega^{\ast}}, {\Z}) + {\Z}$ and ${\Q}_{2}( {\omega^{\ast}}, {\Z}) + {\omega}^{\ast}$.

\section{Coding Trees}

This section introduces coding trees, which carry all the relevant information about lower 1-transitive linear orders.

First, recall that a downwards growing tree $(T,\leqslant)$ is \textbf{Dedekind-MacNeille complete} if its maximal chains are Dedekind-complete in the usual sense,
and if any two incomparable elements have a least upper bound. In fact, this is a special case of a general notion for partial
orders, and the basics are given, for example, in Chapter~7 of \cite{Davey}. Any tree $(T,\leq)$ has a unique (up to isomorphism over $T$) Dedekind-MacNeille completion, that is, a minimal Dedekind-MacNeille complete tree containing it, which is obtained as follows. If $A \subseteq T$ then $A^\mathrm{u}$ denotes the set of upper bounds of $A$ and $A^\mathrm{l}$ the set of lower bounds, that is,
$$A^\mathrm{u}:=\{x \in T : (\forall a \in A) \ (x \geqslant a)  \},\ \mbox{ and}$$
$$A^\mathrm{l}:=\{x \in T : (\forall a \in A) \ (x \leqslant a) \}. \quad$$
A subset $A \neq \emptyset$ is an \textbf{ideal} of $T$ if
 $(A^\mathrm{u})^\mathrm{l}=A$. If $x$ is any vertex of $T$, then the set $\mathrm{I}(x):= \{y \in T : y \leqslant x \}$ is an ideal of $T$. The Dedekind-MacNeille completion of $T$ is the set $\mathrm{I}^D(T)$ of the ideals of $T$ ordered by inclusion. It is easy to see that $T$ embeds in $\mathrm{I}^D(T)$ via the map which takes $x \in T$ to $\mathrm{I}(x) \in \mathrm{I}^D(T)$.
\begin{definition}
If $(T,\leqslant)$ is a downward growing tree, and $x\in T$, then a \textbf{child} of $x$ is some $y$ such
that $y<x$ and there is no $z\in T$ with $y<z<x$. If $x$ is a child of $y$ then $y$ is a \textbf{parent} of $x$.
 We write ${\rm child}(x)$ for the set of children of $x$.
A \textbf{leaf} of $(T,\leqslant)$ is some $x\in T$ such that there is no $y \in T$ with $y < x$. We write ${\rm leaf}(T)$ for the set of leaves of $(T,\leqslant)$.

 A \textbf{levelled tree} is a downward growing tree $(T , \leqslant)$ together with a partition, $\pi$, of $T$ into
 maximal antichains, called \textbf{levels}, such that $\pi$ is linearly ordered by  $\ll$ so that $x \leqslant y$ in $T$ implies
that the level containing $x$ is below the level containing $y$ in the  $\ll$ ordering.

 A \textbf{leaf-branch} $B$ of a (levelled) $(T,\leqslant)$ is a maximal chain of $T$ which contains a leaf.

The supremum of two incomparable points (which exists in the Dedekind-MacNeille completion of $T$,
 even if not in $T$ itself) is called a \textbf{ramification point}.

If $x \in T$ then the relation $\asymp_{x}$ on
$\{ y \in T : y < x \}$ given by 
\[a \asymp_{x} b\ \text{  if there is } c \in T \, \text{ such that } a, b \leqslant c < x\]
 is an equivalence
relation. The equivalence classes are called \textbf{cones} at $x$.
\end{definition}

In the definitions of \textit{right} and \textit{left children} and \textit{coding trees} below, a tree $(T,\leqslant)$ is equipped with a labelling, that is, each vertex is labelled
 by one of the symbols $\Z$, ${\omega}^{\ast}$,
$\Q$, $\dot{\Q}$, ${\Q}_{n}$, ${\dot{\Q}}_n$ ( for $2 \leqslant n \leqslant {\aleph}_{0}$), $\{ 1 \}$ (singleton), or $\mathrm{lim}$.
 Isomorphisms between such trees are required to preserve the labelling.

\begin{definition}\label{A2} Let $x$ be a vertex of $T$ and $\triangleleft$ a linear order on ${\rm child}(x)$. If a vertex is labelled by one of ${\omega}^{\ast}$, ${\dot{\Q}}$ and ${\dot{\Q}}_{n}$, the \textbf{right} child of that vertex is the child which is greatest under the $\triangleleft$ ordering. All the remaining children are \textbf{left} children. If a vertex is labelled by one of $\Z$, $\Q$ or ${\Q}_{n}$, we consider all its
 children to be \textbf{left} children.

The \textbf{left forest} of a vertex is defined to be the partially ordered set consisting of the
left children of the given vertex together with their descendants, with the induced structure of levels and labels.

Two forests are \textbf{isomorphic} provided the subtrees rooted at the 
greatest elements in each forest can be put into one-to-one correspondence in such a way that they are isomorphic as trees.
 \end{definition}
Thus, an isomorphism between two forests preserves the levelling and the labelling, but it is not required to preserve the $\triangleleft$ ordering among children.

\begin{definition}
\label{A1}
A \textbf{coding tree} has the form $( T, \leqslant , \triangleleft, \varsigma , \ll)$ where\\
1. $T$ is a  levelled tree with a greatest element, the root. The tree ordering is $\leqslant$, $\triangleleft$ is a linear ordering on the set of children of each parent and $\ll$ is the ordering of the levels.\\
2. There are countably many leaves.\\
3. Every vertex is a leaf or is above a leaf.\\
4. $T$ is Dedekind-MacNeille complete.\\
5. The vertices are labelled by $\varsigma$, the labelling function, which assigns to the vertices one of the following labels: $\Z$, ${\omega}^{\ast}$,
$\Q$, $\dot{\Q}$, ${\Q}_{n}$, ${\dot{\Q}}_n$ ( for $2 \leqslant n \leqslant {\aleph}_{0}$), $\{ 1 \}$ (singleton), or $\mathrm{lim}$.\\
6. For any two vertices $x_{i}$ and $x_{j}$ on the same level, $\varsigma (x_{i}) \  {\cong}_{l}  \ \varsigma (x_{j})$ or $\varsigma (x_{i}) \  = \ \varsigma (x_{j})$.\\
7. For any vertex $x$ of the tree:\\
if $\varsigma (x) = {\Z}$ or $\Q$ then $x$ has one child;\\
if $\varsigma (x) = {\omega}^{\ast}$ or $\dot{\Q}$ then $x$ has two children;\\
if $\varsigma (x) = {\Q}_{n}$ then $x$ has $n$ children;\\
if $\varsigma (x) = {\dot{\Q}}_n$ then $x$ has $n +1$ children;\\
if $\varsigma (x) = \{ 1 \} $ then $x$ is a leaf and has no children;\\
if $\varsigma (x) = \mathrm{lim}$ then there is only one cone at $x$ ($x$ is not a leaf and has no children).\\
8. At each given level of $T$, the left forests of vertices at that level are all isomorphic in the sense of Definition~\ref{A2}.\\
9. If $x$ is a parent vertex and $y_{0}, y_{1}$ are two of its left children, then the subtrees with roots $y_{0}, y_{1}$ are not isomorphic.
\end{definition}
We will not explain how to define a linear order from  a coding tree until Section~\ref{sec3}, but we illustrate Definition~\ref{A1} in Figure~\ref{Fig1}, where we give the coding trees for the lower 1-transitive linear orders in the lower isomorphism class of $\Z^3$, that is, $\Z^3$, $\omega^*.\Z^2 + \Z^2$, $\omega^*.\Z^2 + \omega^*.\Z +\Z$ and $\omega^*.\Z^2 + \omega^*.\Z +\omega^*$.

\newpage
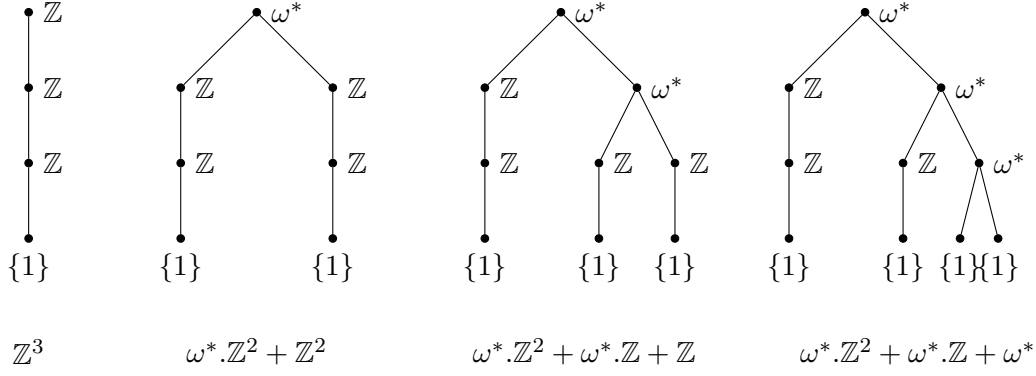
\begin{figure}
\begin{tikzpicture}
  \node [draw,circle,inner sep = 1pt,fill, label=right:$\Z$] (0) at (-6, 2) {};
  \node [draw,circle,inner sep = 1pt,fill, label=right:$\Z$] (1) at (-6, -0) {};
  \node [draw,circle,inner sep = 1pt,fill, label=right:$\Z$] (2) at (-4, -0) {};
  \node [draw,circle,inner sep = 1pt,fill, label=right:$\Z$] (3) at (-4, 1) {};
  \node [draw,circle,inner sep = 1pt,fill, label=right:$\Z$] (4) at (-2, 1) {};
  \node [draw,circle,inner sep = 1pt,fill, label=right:$\Z$] (5) at (-2, -0) {};
  \node [draw,circle,inner sep = 1pt,fill, label=right:${\omega}^{\ast}$] (6) at (-3, 2) {};
  \node [draw,circle,inner sep = 1pt,fill, label=right:$\Z$] (7) at (0, 1) {};
  \node [draw,circle,inner sep = 1pt,fill, label=right:$\Z$] (8) at (0, -0) {};
  \node [draw,circle,inner sep = 1pt,fill, label=right:${\omega}^{\ast}$] (9) at (1, 2) {};
  \node [draw,circle,inner sep = 1pt,fill, label=right:${\omega}^{\ast}$] (10) at (2, 1) {};
  \node [draw,circle,inner sep = 1pt,fill, label=right:$\Z$] (11) at (1.5, -0) {};
  \node [draw,circle,inner sep = 1pt,fill, label=right:$\Z$] (12) at (2.5, -0) {};
  \node [draw,circle,inner sep = 1pt,fill, label=right:$\Z$] (13) at (4, -0) {};
  \node [draw,circle,inner sep = 1pt,fill, label=right:$\Z$] (14) at (4, 1) {};
  \node [draw,circle,inner sep = 1pt,fill, label=right:${\omega}^{\ast}$] (15) at (5, 2) {};
  \node [draw,circle,inner sep = 1pt,fill, label=right:${\omega}^{\ast}$] (16) at (6, 1) {};
  \node [draw,circle,inner sep = 1pt,fill, label=right:$\Z$] (17) at (5.5, -0) {};
  \node [draw,circle,inner sep = 1pt,fill, label=right:${\omega}^{\ast}$] (18) at (6.5, -0) {};
  \node [draw,circle,inner sep = 1pt,fill, label=right:$\Z$] (19) at (-6, 1) {};
  \node [draw,circle,inner sep = 1pt,fill, label=below:$\{1\}$] (21) at (-6, -1) {};
  \node [draw,circle,inner sep = 1pt,fill, label=below:$\{1\}$] (22) at (-4, -1) {};
  \node [draw,circle,inner sep = 1pt,fill, label=below:$\{1\}$] (23) at (-2, -1) {};
  \node [draw,circle,inner sep = 1pt,fill, label=below:$\{1\}$] (24) at (0, -1) {};
  \node [draw,circle,inner sep = 1pt,fill, label=below:$\{1\}$] (25) at (4, -1) {};
  \node [draw,circle,inner sep = 1pt,fill, label=below:$\{1\}$] (26) at (1.5, -1) {};
  \node [draw,circle,inner sep = 1pt,fill, label=below:$\{1\}$] (27) at (2.5, -1) {};
  \node [draw,circle,inner sep = 1pt,fill, label=below:$\{1\}$] (28) at (5.5, -1) {};
  \node [draw,circle,inner sep = 1pt,fill, label=below:$\{1\}$] (29) at (6.25, -1) {};
  \node [draw,circle,inner sep = 1pt,fill, label=below:$\{1\}$] (30) at (6.75, -1) {};
  \node at (-6, -2.5) {$\Z^3$};
   \node at (-3, -2.5) {$\omega^*.\Z^2 + \Z^2$};
  \node at (1.5, -2.5) {$\omega^*.\Z^2 + \omega^*.\Z +\Z\quad$};
  \node at (5.5, -2.5) {$\quad \omega^*.\Z^2 + \omega^*.\Z +\omega^*$};
  \draw  (0) to (1);
  \draw  (6) to (3);
  \draw  (3) to (2);
  \draw  (6) to (4);
  \draw  (4) to (5);
  \draw  (9) to (7);
  \draw  (7) to (8);
  \draw  (9) to (10);
  \draw  (10) to (11);
  \draw  (10) to (12);
  \draw  (15) to (14);
  \draw  (14) to (13);
  \draw  (15) to (16);
  \draw  (16) to (17);
  \draw  (16) to (18);
  \draw  (1) to (21);
  \draw  (2) to (22);
  \draw  (5) to (23);
  \draw  (8) to (24);
  \draw  (11) to (26);
  \draw  (12) to (27);
  \draw  (13) to (25);
  \draw  (17) to (28);
  \draw  (18) to (29);
  \draw  (18) to (30);
\end{tikzpicture}
\caption{Coding trees for lower 1-transitive linear orders in the lower isomorphism class of $\Z^3$} 
\label{Fig1}
\end{figure}


In order to recover a lower 1-transitive linear order from a coding tree, we will need \textit{expanded coding
trees}, which are closely related to coding trees and are defined next. The reason why we need expanded coding trees should become clear in Section~\ref{sec3}. In place of a labelling function on vertices, expanded coding trees carry,  as part of the structure,
a total ordering on the set of children of each vertex, induced by a binary relation~$\triangleleft$. In general, a coding tree and the corresponding expanded coding tree do not have the same vertex set.
 For instance, a point of the expanded
coding tree corresponding to a point labelled $\dot{\Q}$ in the coding tree will have infinitely many children in the expanded coding tree. All the children but the last one are associated with the left child in the coding tree.
The idea is that a lower 1-transitive  linear order  $(X,\leqslant)$ lives on the set of leaves of the expanded coding tree, so the expanded coding tree facilitates the transition
 between coding tree and encoded order.

\begin{definition}
\label{A3}  An \textbf{expanded coding tree} is a structure of the form $( E, \leqslant ,  \ll, \triangleleft)$ where:\\
1. $E$ is a  levelled tree with a greatest element, the root, denoted by $r$. The tree ordering is $\leqslant$, $\ll$ is the ordering of the levels and $ \triangleleft$ is the ordering of the children of each parent vertex.\\
2. $(E,\triangleleft)$ is a partial ordering consisting of a disjoint union of antichains whose elements are exactly the levels of $(E,\leqslant,\ll)$.\\
3. $(E,\leqslant)$ has at most countably many leaves.\\
4. Every vertex of $(E,\leqslant)$ is a leaf or is above a leaf.\\
5. $(E,\leqslant)$ is Dedekind-MacNeille complete.\\
6. If a vertex has any children, then their $\triangleleft$-order type is one of the following; $\Z$, ${\omega}^{\ast}$,
$\Q$, $\dot{\Q}$, ${\Q}_{n}$ or ${\dot{\Q}}_n$ for $2 \leqslant n \leqslant {\aleph}_{0}$.\\
7. Any two vertices $x$ and $x'$ on the same level  are either both parent vertices, or they are both leaves, or
they both  have exactly one cone below them. If $x$ and $x'$ are both parent vertices, then
$({\rm child}(x),\triangleleft) \cong_l ({\rm child}(x'),\triangleleft)$. \\
8. For any parent vertex $x$ of the tree, one of the following holds:
\begin{enumerate}
\item[(i)] the $\triangleleft$-order type of ${\rm child}(x)$ is ${\Z}$, $\Q$, $ {\omega}^{\ast}$ or $\dot{\Q}$ and the left trees rooted at the children of $x$ are all isomorphic, or
\item[(ii)] the children of $x$ are densely ordered by $\triangleleft$ and the trees rooted at the children of $x$
 fall into $n \geqslant 2$ isomorphism classes and this makes them isomorphic to $ {\Q}_{n}$
(for $2 \leqslant n \leqslant {\aleph}_{0}$), or
\item[(iii)] the left children are as in (ii) above, and $x$ has a right child and this makes $\mathrm{child}(x)$ order-isomorphic to ${\dot{\Q}}_n$.
\end{enumerate}
9. At each given level of $E$ the left forests (see Definition~\ref{A2}) from that level are order-isomorphic (meaning that $\leqslant ,  \ll, \triangleleft$ are preserved).
\end{definition}
In 8(ii), we mean that if the elements of ${\rm child}(x)$ are coloured according to the isomorphism type of the trees below
 them, then the corresponding coloured linear order
(with respect to $\triangleleft$) is isomorphic to ${\Q}_n$; likewise in  8(iii).

\section{Construction of a Linear Order from a Coding Tree} \label{sec3}

In this section we describe the relationship between a coding tree and expanded coding tree, and explain how the latter determines a lower 1-transitive linear order. For the coding trees in Figure~\ref{Fig1} it is possible to start either at the root, or at the leaves, and inductively proceed through the tree to determine
 the linear order encoded by it. However, Definition~\ref{A1} does not imply that the levels of a coding tree are well ordered or conversely well ordered. Consider the example in Figure~\ref{Fig2}. 
 
\newpage
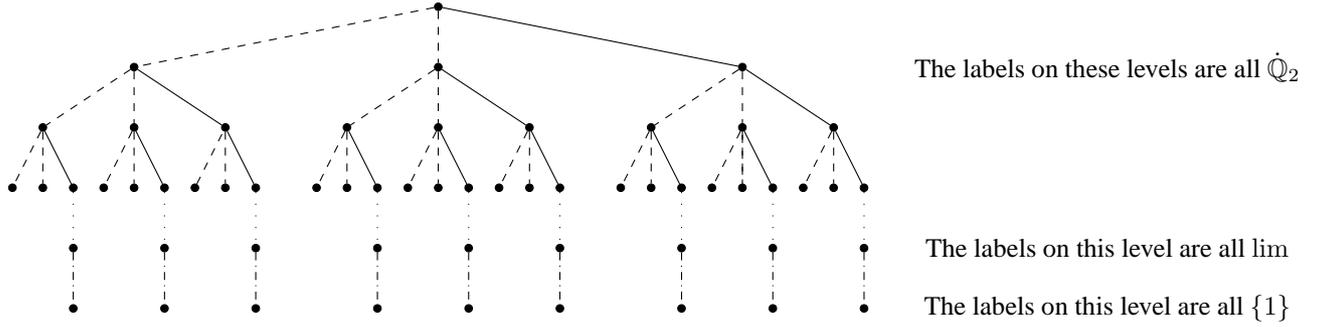
\begin{figure}
\begin{tikzpicture}[scale=0.8]
  \node [draw,circle,inner sep = 1pt,fill] (0) at (0, 3) {};
  \node [draw,circle,inner sep = 1pt,fill] (1) at (0, 2) {};
  \node [draw,circle,inner sep = 1pt,fill] (2) at (0, 1) {};
  \node [draw,circle,inner sep = 1pt,fill] (3) at (1.5, 1) {};
  \node [draw,circle,inner sep = 1pt,fill] (4) at (-1.5, 1) {};
  \node [draw,circle,inner sep = 1pt,fill] (5) at (-0.5, -0) {};
  \node [draw,circle,inner sep = 1pt,fill] (6) at (0, -0) {};
  \node [draw,circle,inner sep = 1pt,fill] (7) at (0.5, -0) {};
  \node [draw,circle,inner sep = 1pt,fill] (8) at (1, -0) {};
  \node [draw,circle,inner sep = 1pt,fill] (9) at (1.5, -0) {};
  \node [draw,circle,inner sep = 1pt,fill] (10) at (2, -0) {};
  \node [draw,circle,inner sep = 1pt,fill] (11) at (-1, -0) {};
  \node [draw,circle,inner sep = 1pt,fill] (12) at (-1.5, -0) {};
  \node [draw,circle,inner sep = 1pt,fill] (13) at (-2, -0) {};
  \node [draw,circle,inner sep = 1pt,fill] (14) at (3, -0) {};
  \node [draw,circle,inner sep = 1pt,fill] (15) at (3.5, -0) {};
  \node [draw,circle,inner sep = 1pt,fill] (16) at (4, -0) {};
  \node [draw,circle,inner sep = 1pt,fill] (17) at (4.5, -0) {};
  \node [draw,circle,inner sep = 1pt,fill] (18) at (5, -0) {};
  \node [draw,circle,inner sep = 1pt,fill] (19) at (5.5, -0) {};
  \node [draw,circle,inner sep = 1pt,fill] (20) at (6, -0) {};
  \node [draw,circle,inner sep = 1pt,fill] (21) at (6.5, -0) {};
  \node [draw,circle,inner sep = 1pt,fill] (22) at (7, -0) {};
  \node [draw,circle,inner sep = 1pt,fill] (23) at (3.5, 1) {};
  \node [draw,circle,inner sep = 1pt,fill] (24) at (5, 1) {};
  \node [draw,circle,inner sep = 1pt,fill] (25) at (6.5, 1) {};
  \node [draw,circle,inner sep = 1pt,fill] (26) at (-3, -0) {};
  \node [draw,circle,inner sep = 1pt,fill] (27) at (-3.5, -0) {};
  \node [draw,circle,inner sep = 1pt,fill] (28) at (-4, -0) {};
  \node [draw,circle,inner sep = 1pt,fill] (29) at (-4.5, -0) {};
  \node [draw,circle,inner sep = 1pt,fill] (30) at (-5, -0) {};
  \node [draw,circle,inner sep = 1pt,fill] (31) at (-5.5, -0) {};
  \node [draw,circle,inner sep = 1pt,fill] (32) at (-6, -0) {};
  \node [draw,circle,inner sep = 1pt,fill] (33) at (-6.5, -0) {};
  \node [draw,circle,inner sep = 1pt,fill] (34) at (-7, -0) {};
  \node [draw,circle,inner sep = 1pt,fill] (35) at (-6.5, 1) {};
  \node [draw,circle,inner sep = 1pt,fill] (36) at (-5, 1) {};
  \node [draw,circle,inner sep = 1pt,fill] (37) at (-3.5, 1) {};
  \node [draw,circle,inner sep = 1pt,fill] (38) at (-5, 2) {};
  \node [draw,circle,inner sep = 1pt,fill] (39) at (5, 2) {};
  \node [draw,circle,inner sep = 1pt,fill] (40) at (2, -1) {};
  \node [draw,circle,inner sep = 1pt,fill] (41) at (2, -2) {};
  \node [draw,circle,inner sep = 1pt,fill] (42) at (0.5, -1) {};
  \node [draw,circle,inner sep = 1pt,fill] (43) at (0.5, -2) {};
  \node [draw,circle,inner sep = 1pt,fill] (44) at (4, -1) {};
  \node [draw,circle,inner sep = 1pt,fill] (45) at (4, -2) {};
  \node [draw,circle,inner sep = 1pt,fill] (46) at (5.5, -1) {};
  \node [draw,circle,inner sep = 1pt,fill] (47) at (5.5, -2) {};
  \node [draw,circle,inner sep = 1pt,fill] (48) at (7, -1) {};
  \node [draw,circle,inner sep = 1pt,fill] (49) at (7, -2) {};
  \node [draw,circle,inner sep = 1pt,fill] (50) at (-1, -1) {};
  \node [draw,circle,inner sep = 1pt,fill] (51) at (-1, -2) {};
  \node [draw,circle,inner sep = 1pt,fill] (52) at (-3, -1) {};
  \node [draw,circle,inner sep = 1pt,fill] (53) at (-3, -2) {};
  \node [draw,circle,inner sep = 1pt,fill] (54) at (-4.5, -1) {};
  \node [draw,circle,inner sep = 1pt,fill] (55) at (-4.5, -2) {};
  \node [draw,circle,inner sep = 1pt,fill] (56) at (-6, -1) {};
  \node [draw,circle,inner sep = 1pt,fill] (57) at (-6, -2) {};
  \node  at (11, 2) {\small The labels on these levels are all ${\dot{\Q}}_2$};
  \node at (11, -1) {\small The labels on this level are all $\mathrm{lim}$};
  \node at (11,-2) {\small The labels on this level are all $\{1\}$};
  \draw  (0) to (39);
  \draw  (39) to (25);
  \draw  (25) to (22);
  \draw  (24) to (19);
  \draw  (23) to (16);
  \draw  (3) to (10);
  \draw  (2) to (7);
  \draw  (4) to (11);
  \draw  (37) to (26);
  \draw  (36) to (29);
  \draw  (35) to (32);
  \draw  (38) to (37);
  \draw  (1) to (3);
  \draw [dashed] (0) to (38);
  \draw [dashed] (0) to (1);
  \draw [dashed] (1) to (2);
  \draw [dashed] (1) to (4);
  \draw [dashed] (2) to (6);
  \draw [dashed] (4) to (12);
  \draw [dashed] (4) to (13);
  \draw [dashed] (39) to (23);
  \draw [dashed] (39) to (18);
  \draw [dashed] (25) to (21);
  \draw [dashed] (25) to (20);
  \draw [dashed] (24) to (17);
  \draw [dashed] (24) to (18);
  \draw [dashed] (28) to (37);
  \draw [dashed] (37) to (27);
  \draw [dashed] (23) to (15);
  \draw [dashed] (23) to (14);
  \draw [dashed] (3) to (9);
  \draw [dashed] (3) to (8);
  \draw [dashed] (2) to (5);
  \draw [dashed] (38) to (35);
  \draw [dashed] (38) to (36);
  \draw [dashed] (35) to (34);
  \draw [dashed] (35) to (33);
  \draw [dashed] (36) to (31);
  \draw [dashed] (36) to (30);
  \draw [loosely dotted] (32) to (56);
  \draw [loosely dotted] (29) to (54);
  \draw [loosely dotted] (26) to (52);
  \draw [loosely dotted] (11) to (50);
  \draw [loosely dotted] (7) to (42);
  \draw [loosely dotted] (10) to (40);
  \draw [loosely dotted] (16) to (44);
  \draw [loosely dotted] (19) to (46);
  \draw [loosely dotted] (22) to (48);
  \draw [dashdotted] (54) to (55);
  \draw [dashdotted] (56) to (57);
  \draw [dashdotted] (52) to (53);
  \draw [dashdotted] (50) to (51);
  \draw [dashdotted] (42) to (43);
  \draw [dashdotted] (40) to (41);
  \draw [dashdotted] (44) to (45);
  \draw [dashdotted] (46) to (47);
  \draw [dashdotted] (48) to (49);
\end{tikzpicture} 
\caption{A coding tree that is neither well founded nor conversely well founded} 
\label{Fig2}
\end{figure}

In this tree, there are infinitely many levels of vertices labelled $\dot{\Q}_2$. The branches are maximal chains which will eventually constantly descend through the right children of $\dot{\Q}_2$. This tree is a coding tree, yet it is neither well founded nor conversely well founded. Examples of this kind are the reason why expanded coding tree are necessary to recover a lower 1-transitive linear order from a coding tree. 

 As in  \cite{TrussGabi}, we start by defining a map which
 associates an expanded coding tree to a coding tree.

\begin{definition}
\label{d1}  Let $( T, \leqslant ,  \varsigma , \ll, \triangleleft)$ be a coding tree, and $( E, \leqslant ,  \ll, \triangleleft)$ be
 an expanded coding tree. We say that $E$ is \textbf{associated} with $T$ via $\phi$ if there is a function $\phi: E \to T$ which
 takes the root $r$ of $E$ to the root of $T$, each leaf of $E$ to some leaf of $T$, and\\
(i) $v_{1} \leqslant v_{2}$ $\implies$ $\phi(v_{1}) \leqslant \phi(v_{2})$,\\
(ii) $\phi$ induces an order-isomorphism from the set of levels of $E$ (ordered by $\ll$) to the set of levels of $T$.
\\
(iii) for each vertex $v$ of $E$, $\phi$ maps $\{ u \in E : u \leqslant v \}$ onto $\{ u \in T : u \leqslant \phi(v) \}$, and
for any leaf $l$ of $E$, $\phi$ maps $[l, r]$ onto $[\phi(l), \phi(r)]$,\\
(iv) for each parent vertex $v$ of $E$, one of the following holds:\\
- $\varsigma(\phi(v)) =$ $\Z$, ${\omega}^{\ast}$, $\Q$, $\dot{\Q}$, and this is
the order type of the children of $v$ under $\triangleleft$; \\
-  $\varsigma(\phi(v)) =$  ${\Q}_{n}$, ${\dot{\Q}}_n$ (for
 $2 \leqslant n \leqslant {\aleph}_{0}$) and for any left children $u$, $u^\prime$ of $v$, if the trees rooted at $u$ and $u^\prime$ are isomorphic then $\phi(u)=\phi(u^\prime)$;\\
- $\varsigma(\phi(v)) = \mathrm{lim}$ if $v$ is
neither a parent nor a leaf (in which case $v$ has just one cone); \\
- $\varsigma(\phi(v)) = \{ 1 \}$ if $v$ is a leaf.

The map $\phi$ is said to be an \textbf{association map} between $T$ and $E$.
\end{definition}


We are now in a position to say explicitly how a tree encodes a linear order. First note
that if $E$ is an expanded coding tree, then there is a natural linear order on ${\rm leaf}(E)$ which we denote by $\triangleleft^*$ and call the \textbf{leaf order}.
If $x,y$ are leaves, we write $x\triangleleft^*y$ if there are $x',y'\in E$ with $x\leqslant x'$, $y\leqslant y'$,
and $x' \triangleleft y'$.

\begin{definition}
\label{d2}
The coding tree $( T, \leqslant ,  \varsigma , \ll, \triangleleft)$ \textbf{encodes} the linear order $(X, \leqslant)$ if there is an expanded coding tree associated with $T$ such that $X$ is (order) isomorphic to the set of leaves of $E$ under the leaf order induced by $\triangleleft$.
\end{definition}

In Theorem~\ref{T1} below, we show how to recover a linear order from a coding tree. In order to do this, we need to define certain functions called \textit{decoding functions}, whose domains are the leaf--branches of a given coding tree and which take a vertex $x$ to an element of the ordered set $\varsigma(x)$.  To cut down to a countable set of functions, even when the coding tree is not well founded or conversely well founded, we begin by choosing arbitrary default values for each of the labels. For each of $\Z$ and $\Q$, there is one default value. In the cases of ${\omega}^{\ast}$, $\dot{\Q}$ and ${\dot{\Q}}_n$, there are two default values, one for the end points and one other. In the case of ${\Q}_{n}$, there are $n$ default values, one of each `colour', whereas $\dot{\Q}_n$ has the same default values as $\Q_n$ plus an additional one for the endpoint.

\begin{definition}
\label{d3}  A  \textbf{decoding function} is a  function $f$  defined on a leaf-branch $B$ of $T$ and such that \\[3pt]
-- the set of non-default values taken by $f$ is finite \\[1pt]
-- for each $x \in B$ with $\varsigma (x) \neq$ $\mathrm{lim}$, $f(x) \in \varsigma (x)$ \\[1pt]
-- if $x$ is a parent vertex and a left child of $x$ is in $\mathrm{dom}\, f$, then $f(x) \neq d_e$, where $d_e$ is the default value for the endpoint \\[1pt]
-- if $x$ is a parent vertex and the right child of $x$ is in $\mathrm{dom}\, f$, then $f(x) = d_e$, \\[1pt]
-- if $\varsigma(x) = \Q_n\,$ or $ \dot{\Q}_n \,$ and $\mathrm{dom}\, f$ contains a left child of $x$ with `colour' $m$, then $f(x)$ has the colour $m$\\[3pt]
If $\varsigma (x) = \, \mathrm{lim}$, we consider $f$ to be undefined at $x$.
\end{definition}

\begin{theorem}
\label{T1}
Any coding tree encodes some linear order, and any two linear orders encoded by the same coding tree are isomorphic.
\end{theorem}

\begin{proof}
We proceed as in \cite{TrussGabi}. Given a coding tree $T$, we construct an expanded coding tree which is associated with $T$ as in Definition~\ref{d1}.

Let $\Sigma_{T}$ be the set of decoding functions on $T$ ordered by $<$, where for $f, g \in \Sigma_T$, $f < g$ if  $f(x_0) < g(x_0)$ with $x_0$ the greatest point for which $f(x) \neq g(x)$. We show that $<$ is a linear ordering. Let $f$ and $g$ be decoding functions such that $f \neq g$. Consider the greatest vertex $x$ such that  $f(x) \neq g(x)$. Such a vertex exists since for $\mathrm{dom}\,f = \mathrm{dom} g$, $f$ and $g$ differ only finitely often. In the case where $\mathrm{dom\,}f \neq \mathrm{dom}\, g$, we appeal to the fact that $T$ is Dedekind-MacNeille complete and therefore it contains all its ramification points, hence there is a vertex $x \in\mathrm{dom}\, f \cap \mathrm{dom}\, g$ such that  $f(x) \neq g(x)$. Then $f(x) < g(x) \Rightarrow f < g$ and  $f(x) > g(x) \Rightarrow f > g$. It is clear this relation is irreflexive and transitive, hence  $(\Sigma_{T}, <)$ is a linear order.

In order for $T$ to encode $(\Sigma_{T}, <)$ according to Definition~\ref{d2}, we must produce an expanded coding tree associated with $T$. Such a tree is given by
$$ E = \{ (x, f\upharpoonright(x, r]) : f \in \Sigma_{T}, x \in \mathrm{dom}\, f \}.$$
The tree ordering is given by letting $(x, f\upharpoonright(x, r]) \leqslant (y, f\upharpoonright(y, r])$ if $x \leqslant y \in \mathrm{dom}\, f$. In addition $(v_{1}, f\upharpoonright(v_{1}, r]) $ is level with $ (v_{2}, g\upharpoonright(v_{2}, r])$ if and only if $v_{1}$ is level with $v_{2}$. It is now clear that $E$ is a levelled tree. Its root is $(r, \varnothing)$. Also, any $(x, f\upharpoonright(x, r])$ lies above a leaf $(l, f\upharpoonright(l, r])$ where $l$ is a leaf in $\mathrm{dom}\, f$.

Each leaf-branch of $E$ is isomorphic to a leaf-branch of $T$, and so it is Dedekind complete. Furthermore, since $T$ contains all its ramification points, so does $E$, and therefore $E$ is Dedekind-MacNeille complete.

Now we consider the possible order types of sets of children of vertices of $E$. Let $(x, f\upharpoonright(x, r])$ be a parent vertex in $E$. Then $x$ is a parent vertex in $T$. The order type of the children of  $(x, f\upharpoonright(x, r])$ is determined by 
\[\{(x, g\upharpoonright[x, r]) : g \in \Sigma_{T}, x \in \mathrm{dom}\,g,  f\upharpoonright(x, r] = g\upharpoonright(x, r] \}\, . \]
Since $x$ is a parent vertex, $f(x) \in \varsigma(x)$. Hence the order type of the children of $(x, f\upharpoonright(x, r])$ is equal to the label of $x$ in $T$. In the case of $\varsigma(x) = {\Q}_{n}$ or ${\dot{\Q}}_n$ we may say that the `coloured' order type of the children in $E$ is ${\Q}_{n}$ or ${\dot{\Q}}_n$.\\
If $(x, f\upharpoonright(x, r])$ is neither a parent vertex nor a leaf, then $x$ is neither a parent nor a leaf, and so $x$ is labelled $\mathrm{lim}$.

The mapping $\phi$ is given by  $\phi((x, f\upharpoonright(x, r])) = x$. This preserves root, leaves and, as we have just seen, it preserves the relation between labels of vertices in $T$ and the (coloured) order type of the children of those vertices in $E$. Also $x \leqslant y \Rightarrow \phi(x) \leqslant \phi(y)$ and it is clear that for each vertex $x$ of $E$, $\phi$ maps $\{ u \in E : u \leqslant x \}$ onto $\{ u \in T : u \leqslant \phi(x) \}$ and for any leaf $l$ of $E$, $\phi$ maps $[l, r]$ onto $[\phi(l), \phi(r)]$. Therefore $E$ is associated with $T$ and $\Sigma_{T}$ is order isomorphic to the set of leaves of $E$. Hence $T$ encodes $\Sigma_{T}$.

A back-and-forth argument shows that any two countable linear orders encoded by the same coding tree $( T, \leqslant ,  \varsigma , \ll, \triangleleft)$ are isomorphic.
\end{proof}

\begin{theorem}
\label{T3}
The ordering $\Sigma_{T}$ encoded by the coding tree $( T, \leqslant ,  \varsigma , \ll, \triangleleft)$ is countable and lower 1-transitive.
\end{theorem}

\begin{proof}
 The way in which  $\Sigma_{T}$ has been defined ensures that it is countable.

We now show that $\Sigma_{T}$ is lower 1-transitive. Take any $f, g \in \Sigma_{T}$ and consider the initial segments $( -\infty , f]$ and $( -\infty , g]$. Now $\Sigma_{T}$ is defined to be the set of all functions on the leaf-branches of $T$ which take a default value at all but finitely many points. By definition of the ordering on $\Sigma_{T}$, an initial segment of $\Sigma_{T}$ at $f$ can be written as
$$( -\infty , f] = \{ f \}  \cup \{ p \in \Sigma_{T} :  (\exists x \in \mathrm{dom}\, f )( p(x) < f(x)) \wedge (\forall y > x)( p(y) = f(y)) \}.$$
Let $L_{i}$ be the $i$th level of the tree, and let
$$\Gamma^{f}_{i} = \{ p \in ( -\infty , f] : p(x^{f}_{i}) < f(x^{f}_{i}) \wedge (\forall y > x^{f}_{i})( p(y) = f(y)) \},$$
where $x^{f}_{i}$ denotes the element of $\mathrm{dom}\, f$ on the level $L_{i}$, and $x_{i}$ denotes the element of $\mathrm{dom}\, p$ in $L_{i}$, where $p$ is a typical member of $\Sigma_{T}$.)

Then, by definition of the $\ll$ - ordering of the levels, it is clear that $( -\infty , f]$ is the disjoint union of all the $\Gamma^{f}_{i}$, and furthermore that $i \ll j \Rightarrow \Gamma^{f}_{i} > \Gamma^{f}_{j}$ (where this means that every element of $\Gamma^{f}_{i}$ is greater than every element of $\Gamma^{f}_{j}$). Since the same is true of the $\Gamma^{g}_{i}$, to show that $( -\infty , f] \cong ( -\infty , g]$, it suffices to show that $\Gamma^{f}_{i} \cong \Gamma^{g}_{i}$ for each $i$, and the desired isomorphism from  $( -\infty , f]$ to $ ( -\infty , g]$ is obtained by patching together all the individual isomorphisms.

We remark that for $\varsigma (x^{f}_{i}) = \, \mathrm{lim}$, we have $\Gamma^{f}_{i} = \varnothing$. The label $\mathrm{lim}$ is not a linear order, so by condition 3 of Definition~\ref{A1}, if a vertex on level $i$ is labelled $\mathrm{lim}$, then all vertices on level $i$ are labelled $\mathrm{lim}$. This shows that when $i$ is  a level with vertices labelled $\mathrm{lim}$, we have $\Gamma^{f}_{i} \cong \Gamma^{g}_{i}$. 

We now consider the cases where the vertices on level $i$ are not labelled $\mathrm{lim}$. There is an isomorphism $\varphi$ from $( - \infty , f(x^{f}_{i})] \cap \varsigma (x^{f}_{i}) $ to $( - \infty , g(x^{g}_{i})] \cap \varsigma (x^{g}_{i})$. Moreover, there is an isomorphism $\psi$ between the left forests at the points $x_i^f$ and $x_i^g$. Now let $x_j$ be the member of $\mathrm{dom}\,p$ at the level $L_j$ for a typical $p$. We now define
\begin{displaymath}
\Phi_{i} (p) (\psi (x_j))=  \left\{ \begin{array} {ll}
 g(x^{g}_{j}) & \textrm{ if } j > i\\
\varphi (p(x_{j})) & \textrm{ if } j = i\\
p(x_{j}) & \textrm{ if } j < i\\
\end{array} \right.
\end{displaymath}
where $p \in \Gamma^{f}_{i}$.

We must now show that $\Gamma^{f}_{i}$ is mapped 1-1 into $\Gamma^{g}_{i}$ by $ \Phi_{i}$. This gives our result. We have that $ \Phi_{i}(p) \in \Sigma_{T}$, because all such $ \Phi_{i}(p)$ are defined on leaf-branches of $T$ and they take a default value at all but finitely many points, since both $ g(x^{g}_{j})$ and $p(x_{j})$ take the default value at all but finitely many points (possibly with $\varphi (p(x_{j}))$ in addition).

It is easy to see that $\Phi_i$ is surjective. For injectivity, suppose $ \Phi_{i}(p_{1}) =  \Phi_{i}(p_{2})$. Then since $p_{1} , p_{2} \in \Gamma^{f}_{i}$, $p_{1} (x_{j}) = p_{2} (x_{j}) = f(x^{f}_{j})$ for all $j > i$ and $p_{1} (x_{j}) = p_{2} (x_{j})$ for $j < i$ by the third clause. Since $\varphi$ is an isomorphism, $\varphi (p_{1} (x_{i})) = \varphi (p_{2} (x_{i}))$ implies that $p_{1} (x_{j}) = p_{2} (x_{j})$. Hence $p_{1} (x_{j}) = p_{2} (x_{j})$ for all $j$.\\
\end{proof}

\section{Construction of a Coding Tree from a Linear Order}

In this section we complete the classification by showing that  any countable and lower 1-transitive $(X, \leqslant)$ is encoded by a suitable coding tree. We first find the associated expanded coding tree for $(X, \leqslant)$. This is done by building a tree of \textit{invariant partitions} of $(X, \leqslant)$ in the sense of Definition~\ref{D2} below. We show this is in fact an expanded coding tree for $(X, \leqslant)$ and then give the association map between them.

\begin{definition}
\label{D2}
An \textbf{invariant partition} of $X$ is a partition $\pi$ into convex subsets, called \textbf{parts}, which is invariant under lower isomorphims of $(X, \le)$ into itself. That is, for any $a, b \in X$, any order isomorphism $f : ( -\infty , a] \to ( - \infty , b]$, and any $x, y \leqslant a$,
$$x \sim_{\pi} y \iff f(x) \sim_{\pi} f(y).$$
\end{definition}
The family $I$ of all parts of invariant partitions of $X$ is partially ordered by inclusion. This allows us to define a levelled tree structure on $I$.
\begin{definition}
\label{D3}
For a lower 1-transitive linear order $(X, \le)$, the \textbf{invariant tree} associated with $X$ is the levelled tree $I$ whose vertices are parts in the invariant partitions of $X$ ordered by $\subseteq$ in such a way that
\begin{itemize}
\item[(i)] $X \in I$ is the root
\item[(ii)] each level is an invariant partition of $X$
\item[(iii)] the leaves are the singletons $\{ x \}$ for $x \in X$
\item[(iv)] every invariant partition of $X$ into convex subsets of $X$ is represented by a level of vertices in $I$.
\end{itemize}
\end{definition}
We remark that $I$ has a root since $X$ is itself lower 1-transitive and a convex subset of $X$. Moreover, the parts of any invariant partition of $X$ are lower isomorphic and lower 1-transitive. Lemmas~\ref{F2} and~\ref{F3} show that for any countable, lower 1-transitive linear order, the family $I$ is a levelled tree, thereby justifying the description \emph{the} invariant tree. The proof of Lemma~\ref{F2} is left to the reader.

\begin{lemma}
\label{F2}
If $(X, \leqslant)$ is a countable lower 1-transitive linear order and $\pi$ is an invariant partition of $X$, then  $X/_{\sim_{\pi}}$ is also a countable lower 1-transitive linear order with the ordering induced by $(X, \leqslant)$.
\end{lemma}


\begin{definition} Let $\pi_i, \pi_j$ be invariant partitions of $(X, \leqslant)$. We say that $\pi_{i} $ is a \textbf{refinement} of $\pi_{j}$ if every element of $\pi_{j}$ is a union of members of $\pi_{i}$.\end{definition}

\begin{lemma}
\label{F3}
Given any two nontrivial invariant partitions $\pi_{1} , \pi_{2}$ of $X$ into convex subsets of $X$, one is a refinement of the other, and moreover $\pi_1$ and $\pi_2$ have no part in common.
\end{lemma}

\begin{proof}
Let $\sim_1, \sim_2$ be the equivalence relations defining $\pi_{1} , \pi_{2}$ respectively. 
We want to show that
$$ (\forall x , y \in X)(x \sim_1 y \Rightarrow  x \sim_2 y) \vee (\forall x , y \in X)(x \sim_2 y \Rightarrow  x \sim_1 y). $$
Suppose both disjuncts are false. Then there are $x, y, u, v$ such that
\begin{itemize}
\item $x \sim_{1} y$ and  $x \nsim_{2} y$, and
\item $u \nsim_{1} v$ and $u \sim_{2} v$.
\end{itemize}


We may assume that $x < y$ and $u < v$. Let $f : (- \infty , y] \to (- \infty , v]$ be an isomorphism. Then
$f(x) < v$ and $f(x) \sim_1 v$. Moreover, we must have $u < f(x)$, otherwise  $u \sim_1 v$ by convexity. So
$u < f(x) < v$, and therefore $f(x) \sim_2 v$ by convexity. However, $x \nsim_2y$ implies $f(x) \nsim_2 v$, which is a contradiction.

Without loss of generality, assume that $\pi_1$ is a refinement of $\pi_2$. We want to show that $\pi_1 \cap \pi_2 = \varnothing$. So suppose for a contradiction that there is $p \in \pi_1 \cap \pi_2$, and let $x, y \in X$ be such that $x \nsim_1y$ and $x \sim_2y$. Pick $z \in p$ and let
$g : (- \infty , y] \to (- \infty , z]$ be an isomorphism. Then $g(x) \sim_2 z$, and since $p \in \pi_1 \cap \pi_2$, we have $g(x) \sim_1 g(y)$, contradicting our choice of $x$ and $y$.

\end{proof}
The next lemma proves the Dedekind-MacNeille completeness of the invariant tree.
\begin{lemma}
\label{F4}
The invariant tree $I$ of a lower 1-transitive linear order $(X, \leqslant)$ is Dedekind-MacNeille complete.
\end{lemma}

\begin{proof} We need to show  that
\begin{itemize}
\item[(i)] the supremum of any two vertices in $I$ is also a vertex in $I$, and
\item[(ii)] every descending chain of vertices in the tree which is bounded below has an infimum in the tree.
\end{itemize}

To show (i), consider two vertices $p_1, p_2 \in I$ that are parts of two partitions  $\pi_{1}, \pi_{2}$, respectively. Without loss of generality, assume that $\pi_{1}$ refines $\pi_{2}$. Then either $p_{1} \subseteq p_{2}$ (and $p_{2}$ is the supremum of $p_{1}$ and $p_{2}$) or $p_{1} \subseteq p'_{2} \in \pi_{2}$. So this problem reduces to showing that the supremum of any two vertices on the same level is in $I$.

We know that  $p_{2} ,  p'_{2}  \subseteq p$ with $p \in \pi$, for some $\pi \in I$ which coarsens $\pi_{2}$ --- for instance $\{ X \}$ itself. Let $\sim_\pi$ be the equivalence relation corresponding to $\pi$. Then $a \sim_\pi b$ for $a, b \in p_{2} ,  p'_{2}$ respectively. 

Consider the set of partitions $\pi^\prime$ that refine $\pi$ for which $a \sim_{\pi^\prime} b$, where $\sim_{\pi^\prime}$ is the corresponding equivalence relation.  By Lemma~\ref{F3} this is a descending chain of partitions. If the set of parts containing both $a$ and $b$ has an infimum, then $p_{2} ,  p'_{2}$ have a supremum. So the verification of (i) reduces to that of (ii).

For (ii), consider a descending chain of vertices $p_{\gamma}$ that are parts of a descending chain of partitions $\pi_{\gamma}$ bounded below by $p$, say, where $p \neq \varnothing$.  Let $\sim_{\gamma}$ be the equivalence relation corresponding to $\pi_{\gamma}$. Then define $x \sim y$ if $ x \sim_{\gamma} y$ for all $\gamma$. Let $f$ be a lower isomorphism of $(X, \leqslant)$. Then $x \sim y$ implies $f(x) \sim_{\gamma} f(y)$ for all $\gamma$ because each of the $\sim_{\gamma}$ is an invariant relation. Hence $f(x) \sim f(y)$ and so $\sim $ is an invariant relation. If $\pi$ is the corresponding partition, then $\pi$ is a  partition into lower 1-transitive, lower isomorphic convex subsets of $X$, and so its parts are vertices in $I$. Then $p$ is contained in some member of $p^\prime$ of $\pi$, and $\pi^\prime$ is the infimum of the $p_{\gamma}$.
\end{proof}

\begin{theorem}
\label{T4}
The invariant tree $I$ of a lower 1-transitive linear order $(X, \leqslant)$ is an expanded coding tree whose leaves are order-isomorphic to $(X, \leqslant)$.
\end{theorem}
\begin{proof}
Firstly, the leaves of $I$ are singletons containing the elements of $X$, and so they are isomorphic to $X$.

Definition~\ref{D3} ensures that $I$ is a levelled tree whose root is $X$. The tree ordering is containment, the ordering of the levels is the one induced by $\subseteq$ on the set of invariant partitions of $X$, and the ordering of the children of a parent vertex is the one induced by the linear order on $X$. Since $X$ is countable, $I$ has countably many leaves. It is clear that every vertex of $I$ is a leaf or is above a leaf. So conditions 1 to 4 of Definition~\ref{A3} are satisfied. Moreover, $I$ is Dedekind-MacNeille complete by Lemma~\ref{F4}.


In order to verify condition 6 of Definition~\ref{A3}, we need to show that the order type of the children of a parent vertex in $I$ is one of $\Z$, ${\omega}^{\ast}$, $\Q$, $\dot{\Q}$, ${\Q}_{n}$ or ${\dot{\Q}}_n$ ( for $2 \leqslant n \leqslant {\aleph}_{0}$). Consider a successor level $\pi_{i+1}$ of $I$, so $\pi_{i}$ is the predecessor. Let $p \in \pi_{i + 1}$. Then $p$ is lower 1-transitive, and the children of $p$ are those elements of $\pi_{i}$ which are convex subsets of $p$. These children are lower 1-transitive linear orders and are lower isomorphic to each other. Let $\sim_{\pi_i}$ be the equivalence relation that defines $\pi_{i}$. Then, by Lemma~\ref{F2}, $p / \sim_{\pi_i}$ is also lower 1-transitive, and the order type of $p / \sim_{\pi_{i}}$ tells us how the children of $p$ are ordered. In order to describe the possible order types, we look at the structure forced by a specific invariant equivalence relation, namely, the relation $\sim_{\mathrm{fin}}$ that identifies points that are finitely far apart,  defined by
\[ x \sim_{\mathrm{fin} } y \ \mbox{ iff } \ x \leqslant y \ \mbox{ and } [x, y] \ \mbox{ is finite, or } \ y \leqslant x \ \mbox{ and } \ [y, x] \ \mbox{ is finite. }\]

For any linear order, the equivalence classes of $\sim_{\mathrm{fin}}$ must be either finite, $\omega$, ${\omega}^{\ast}$ or ${\Z}$. If $(X, <)$ is lower 1-transitive, the equivalence classes of this form are either singletons,  ${\omega}^{\ast}$, or ${\Z}$. If one equivalence class is a singleton, then they all are, and then the ordering is dense with no least endpoint. Hence it is isomorphic to $\Q$ or $\dot{\Q}$.

Since $p / \sim_{\pi_{i}}$ is a lower 1-transitive linear order, we can take its quotient by $\sim_{\mathrm{fin}}$. There are two cases.

Case 1: the equivalence classes of $(p / \sim_{\pi_{i}}) / \sim_{\mathrm{fin} }$ are non-trivial. Then, by the maximality of $I$, there can be only one equivalence class, that is, $p / \sim_{\pi_{i}}$ itself. If there is no last child, then $p$ is equal to $\Z$ copies of its children; otherwise, the order type of $p / \sim_{\pi_{i}}$ is ${\omega}^{\ast}$.

Case 2: the equivalence classes of $(p / \sim_{\pi_{i}}) / \sim_{\mathrm{fin} }$ are trivial.
Then the parts of $\pi_{i}$ are dense within $p$. We aim to show that $p / \sim_{\pi_{i}}$ is a ${\Q}$, $\dot{\Q}$, ${\Q}_{n}$ or ${\dot{\Q}}_{n}$ combination of its children.

If all the left children of $p$ are isomorphic, then $\mathrm{child}(p)$ is isomorphic to $ {\Q}$, or $\dot{\Q}$ if the right child exists.


If not all the left children of $p$ are isomorphic, then we show that $\mathrm{child}(p)$ is isomorphic to $ {\mathbb Q}_n$, or ${\dot{\mathbb Q}}_n$ if $p$ has a
right child, where the set $\Gamma$ of (colour, order-)isomorphism types of the left children of $p$ has size $n$. Suppose, for a contradiction,
that $p$ is not the ${\mathbb Q}_n$ mixture of its children. 
Then there are two elements of $\Gamma$ such that not all other elements of $\Gamma$ occur between them in $p$. Let $\gamma$ be a member of $\Gamma$ which does not occur
between all pairs, and let us define $\sim$ on $\pi$ by $y \sim z$ if $y = z$, or if no point of $[y, z]$ (or $[z, y]$ if $z < y$) has isomorphism type
$\gamma$. This is an invariant partition of $\pi$ into convex pieces, and is proper and non-trivial, which contradicts $\pi_i$ and $\pi_{i+1}$ being on consecutive levels.

This verifies condition 6 of Definition~\ref{A3} for a parent vertex on a successor level of the invariant tree $I$.

Now consider the levels which are not successor levels. Firstly, this includes the trivial partition, $\pi_{0}$, given by the relation $ x \sim_{\pi_{0}} y \iff x = y$. These vertices are leaves. 

There remains the case of vertices which do not have children in $I$. If one part of an invariant partition does not have a child then, clearly, none of them do. Dedekind-MacNeille completeness implies that these vertices have one cone below them.

For condition 7, let $x$ and $ x'$ be two vertices of $I$ on the same level. Then $x, x^\prime$ are parts of an invariant partition, so either they are both parents, or they are both leaves, or both are neither of these, in which case they have a single cone below them. Moreover, if $x,  x'$ are both parent vertices, then $({\rm child}(x),\triangleleft)$ is lower-isomorphic to $({\rm child}(x'),\triangleleft)$, since $(X, \leqslant)$ is lower 1-transitive and $x$ and $ x'$ are parts of an invariant partition.


For condition 8, let $x \in I$ be a parent vertex. Suppose that $({\rm child}(x), \triangleleft) \cong {\Q}$, $\dot{\Q}$, ${\Q}_{n}$ or ${\dot{\Q}}_n$. Here two children vertices $a, b$ have the same colour when they are isomorphic. This isomorphism induces an isomorphism on the trees rooted at $a, b$.
If $({\rm child}(x), \triangleleft) \cong {\Z}$, we wish to show the children of $x$ are all isomorphic, and hence the trees below the children are isomorphic. Now, the children of $x$ are all a finite distance apart. In particular, each child has a successor and a predecessor. If $a$ and $b$ are children of $x$, the existence of an isomorphism from the successor of $a$ to the successor of $b$ implies that $a$ and $b$ are isomorphic. The argument in the case $({\rm child}(x), \triangleleft) \cong {\omega}^{\ast}$ is similar.

Finally we show $I$ satisfies condition 9 of~Definition~\ref{A3}.

Let $x$ and $y$ be distinct vertices on the same level. If $x$ and $y$ have no children, the condition holds trivially. So suppose that $x$ and $y$ are parent vertices and let $a \in x$ and $b \in y$. By lower 1-transitivity, there is an isomorphism $\varphi : (- \infty , a] \to (- \infty , b]$ 
which induces an isomorphism between $(- \infty , a] \cap x$ and $(- \infty , b] \cap y$. Let $x_{a}, y_{b}$ be the children of $x, y$ containing $a, b$ respectively. Then  $(- \infty , a] \cap x_{a}$ and $(- \infty , b] \cap y_{b}$ are isomorphic. Consider the sets $\Gamma_{a}$, $\Gamma_{b}$ of children of $x, y$ to the left of  $x_{a}, y_{b}$ respectively.  Since $\varphi (\Gamma_{a}) = \Gamma_{b}$, the left forests of $x$ and $y$ are isomorphic.  Since $a, b$ are arbitrary,  $\Gamma_{a}$ and $  \Gamma_{b}$ can contain any particular left children of $x$ and $y$.
\end{proof}

We now must show how to construct a coding tree for $(X, \leqslant)$ given the invariant tree $I$, and give an inverse association map between them.\\
Informally, the coding tree is obtained from $I$ by amalgamating left children who are siblings and whose trees of descendants are isomorphic. The parent vertex is then labelled according to the order type of its children in $I$.\\
For each level $s$ of $I$ we define a relation $\simeq_{s}$ on $I$ that tells us which vertices to amalgamate:\\[3pt]
$x \simeq_{s} y \text{ if there are } x' \supseteq x, \ y' \supseteq y$ such that
\begin{itemize}
\item[(i)] the tree of descendants of $x'$ is isomorphic to the tree of descendants $y'$ under $\theta$, a suitable  order isomorphism (respecting both the $\leqslant$ order and the $\vartriangleleft$ order)
\item[(ii)] $x', y'$ are left children of a vertex $z$ and lie on level $s$, or $x' = y'$
\item[(iii)] $\theta(x) = y$.
\end{itemize}
Note that the clauses guarantee that $x, y$ are level. Now we define a relation $\simeq$ on the whole of $E$ as follows:
 \[x \simeq y \iff \exists x = x_{0},  \ldots , x_{n} = y, \text{ where for each } i=0, \ldots, n-1 \text{ there is } s_i \text{ with } x_{i} \simeq_{s_i}   x_{i + 1}.\]
The relation $\simeq$ is an equivalence relation on $I$, and $T$ is then the set of equivalence classes on $I$, labelled as described above. We denote an element of $T$ by $[x]$, where $x \in I$. The next lemma ensures that the ordering on $I$ induces one on $T$.

\begin{lemma}\label{lemma48}
Let $[x],\,[y]\in T$ be such that $x \leqslant y$ (in $I$), and let $x^\prime \in [x]$. Then there is $y^\prime \in [y]$ such that $x^\prime \leqslant y^\prime$.
\end{lemma}
\begin{proof}
Let $x, y$ and $x^\prime$ be as in the statement. Since $x^\prime \in [x]$,  there are $u, v$ and $w$ in $I$ such that $u, v$ are left children of $w$ and $x \leq u$, $x^\prime \leqslant v$. Moreover, the tree of descendants of $u$ is isomorphic to the tree of descendants of $v$ by an isomorphism $\theta$ such that $\theta(x)=x^\prime$. Now, either $y \geqslant w$ or $y <w$. If $y \geqslant w$ then $x^\prime < w \leqslant y$, so $y$ is the required $y^\prime$. If $y<w$, then  $y \leqslant u$. Then $x^\prime = \theta(x) \leqslant \theta(y)$ and, since $\theta(y) \leqslant v$, this implies that $x^\prime \leqslant v$. But $\theta(y) \in [y]$ because of the way $\simeq$ is defined, so $\theta(y)$ is the required $y^\prime$. 
\end{proof}

\begin{theorem}\label{t49}
The set of $\simeq$-classes on the invariant tree of $(X, \leqslant)$ is a coding tree for $(X, \leqslant)$.
\end{theorem}

\begin{proof}
Let $T$ be the family of $\simeq$-equivalence classes on $I$. Let $[x], [y] \in T$ and define
\[ [x] \leqslant [y] \iff (\exists x^\prime \in [x]) (\exists y^\prime \in [y]) (x^\prime \leqslant y^\prime) \ \text{ (in } I).\]
Lemma~\ref{lemma48} ensures that $\leqslant$ is well defined and transitive, so $\leqslant$ is an order. 

Since $\leqslant$ is the order induced by that on $I$, $T$ is a tree with root $[r]$ and, since $\simeq$ is level preserving, $T$ is a levelled tree. Moreover, $T$ is countable, and every vertex of $T$ is a leaf or is above a leaf. We verify Dedekind-MacNeille completeness. Firstly note that all leaf-branches of $T$ are isomorphic to some leaf-branch of $I$ and so the leaf-branches of $T$ are Dedekind complete. We must now show that the least upper bound of any two vertices $[x], [y] \in T$ is in $T$. Since $I$ is Dedekind-MacNeille complete, any $x^\prime \in [x]$ and $y^\prime \in [y]$ have a least upper bound in $I$. Let
\[ \Gamma = \{ z \in I : z \text{ is the least upper bound of } x^\prime \text{ and } y^\prime \text{ for some } x^\prime \in [x], y^\prime \in [y] \}. \]
If $z \in \Gamma$, then $[x^\prime]=[x] \leqslant [z]$ and $[y^\prime] = [y] \leqslant [z]$ for some $x^\prime, y^\prime$, and so $[z]$ is an upper bound for $[x]$ and $[y]$. Now let
$ \Gamma^\prime = \{ [z] : z \in \Gamma \}$.
Since $\Gamma^\prime$ contains the upper bounds of $[x]$ and $[y]$, it is linearly ordered. Moreover, it is bounded above by $[r]$ and below by $[x]$. Let $\Gamma$ be the chain 
\[ \ldots [z_{-n}] \geqslant \ldots \geqslant [z_{0}] \geqslant [z_{1}] \geqslant \ldots\ [z_n] \geqslant \ldots \]
By Lemma~\ref{lemma48}, for any $u \in [z_{i+1}]$ there is $v \in [z_i]$ such that $u \leqslant v $. Hence we can construct a corresponding chain of vertices in $I$. 
If the $[z_{i}]$ do not have an infimum, then there is a chain of vertices in $I$ bounded below by $x \in [x]$ and  without an infimum. This contradicts the Dedekind-MacNeille completeness of $I$. Then the infimum of $\Gamma^\prime$ is the least upper bound of $[x]$ and $[y]$.

Next we examine the labelling. Suppose $x  \in I$ is a parent vertex. Then $[x] \in T$ is also a parent vertex and we let $\varsigma([x]) = ({\rm child}(x), \triangleleft)$, the order type of the children of $x$ in $I$. This is well defined, as $x \simeq y$ implies that $x$ and $y$ are isomorphic and hence the sets of their children have the same order type. Now, since $({\rm child}(x), \triangleleft)$ is one of  $\Z$, ${\omega}^{\ast}$,
$\Q$, $\dot{\Q}$, ${\Q}_{n}$, ${\dot{\Q}}_n$ ( for $2 \leqslant n \leqslant {\aleph}_{0}$), it follows that $\varsigma([x])$ is also one of the above.\\
If $x$ is neither a parent nor a leaf, then neither is $[x]$. Hence we label $[x]$ by $\mathrm{lim}$. The leaves are labelled $\{ 1 \}$.\\
Let $[x], [y] \in T$ be level parent vertices and let $x, y \in I$ be representatives. Then  $\varsigma ([x]) \  {\cong}_{l}  \ \varsigma ([y])$ follows from the fact that  $({\rm child}(x), \triangleleft) \  {\cong}_{l}  \ ({\rm child}(y), \triangleleft)$.\\
When  $[x], [y] \in T$ are level but neither parent vertices nor leaves (if $[x]$ is not a parent vertex and  $[x]$ are $[y]$ level, then $[y]$ is not a parent vertex), both are labelled $\mathrm{lim}$, as remarked earlier. Hence  $\varsigma ([x]) = \varsigma ([y])$ as required. The case when $[x], [y] \in T$ are leaves is similar.\\
We now show that $T$ fulfils condition 7 of Definiton~\ref{A1}. The number of children of $[x] \in T$ is the number of equivalence classes of the children of vertices $x^\prime \in [x]$ in $I$.
We consider various cases.\\
Case 1: $({\rm child}(x), \triangleleft) \cong {\Z}, {\Q}$\\
All the children of $x$ are left children. We have also seen that they are all isomorphic and hence they are all $\simeq$-equivalent. Therefore there is one equivalence class below $[x]$.\\
Case 2: $({\rm child}(x), \triangleleft) \cong {\omega}^{\ast}, \dot{\Q}$\\
Again all the left children of $x$ are isomorphic and hence they are all $\simeq$-equivalent. A right child of $x$ forms its own equivalence class under $\simeq$. In these cases $[x]$ has two children.\\
Case 3. $({\rm child}(x), \triangleleft)\cong{\Q}_{n}$, ${\dot{\Q}}_n$.\\
The `colours' are the isomorphism types of the children of $x$ in $I$. There are $n$ isomorphism types amongst the left children. The left children which are isomorphic are also $\simeq$-equivalent. Hence there are $n$ ($n + 1$ in the case of ${\dot{\Q}}_n$) $\simeq$-classes below $[x]$.\\
Clause 8 of Definition~\ref{A1} follows from the corresponding fact about the expanded coding tree. Given two order isomorphic forests in the expanded coding tree, clearly the $\simeq$-classes on two such forests are also isomorphic.\\
Finally, since $\simeq$ amalgamates isomorphic trees of descendants of sibling left vertices, the tree of descendants of two sibling vertices in the resulting $T$ will not be isomorphic.
\end{proof}

\noindent We have obtained a coding tree from $(X, \leqslant)$. We have now to show that this tree does encode $(X, \leqslant)$.

\begin{theorem}
\label{C2}
The coding tree, $(T, \leqslant, \vartriangleleft, \varsigma, \ll)$ obtained from $(X, \leqslant)$ encodes $(X, \leqslant)$ in the sense of Definition~\ref{A3}.
\end{theorem}

\begin{proof}
Firstly we show that the expanded coding tree $I$ of invariant partitions of $X$ is associated with $T$ in the sense of Definition~\ref{d1}. The association function $\phi \rightarrow T$ is defined by $\phi(x)=[x]$, and the labelling function on $T$ is defined as follows:
\begin{itemize}
\item[(i)] if $x$ is a parent vertex, the label of $\phi(x)$ is equal to $({\rm child}(x), \triangleleft)$, the (coloured) order type of the children of $x$ in $I$,
\item[(ii)] if $x$ is neither a parent nor a leaf, the label of $\phi(x)$ is $\mathrm{lim}$,
\item[(iii)] if $x$ is a leaf, the label of $\phi(x)$ is $\{ 1 \}$.
\end{itemize}
As remarked in the proof of Theorem~\ref{t49}, this labelling is well defined. Moreover, the labels satisfy condition (iv) of Definition~\ref{d1}. By the way $T$ is constructed, it is clear that $\phi$ preserves levels. Moreover, the ordering on $T$ is such that $x \leqslant y$ in $I$ implies that $\phi(x) \leqslant \phi(y)$ in $T$. This ensures that conditions (i), (ii) and (iii) of Definition~\ref{d1} are satisfied.

The construction of $I$ ensures that $X$ is order isomorphic to the set of leaves of $I$. Therefore $T$ encodes the linear order $X$ in the sense of Definition~\ref{d2}, as required.

\end{proof}

Therefore  $(T, \leqslant, \varsigma, \ll, \triangleleft)$ encodes $(X, \leqslant)$. Furthermore the set of $\simeq$-classes on the expanded coding tree $( E, \leqslant ,  \ll, \triangleleft)$ constructed from the coding tree $(T, \leqslant, \varsigma, \ll, \triangleleft)$ is isomorphic to $(T, \leqslant, \varsigma, \ll, \triangleleft)$, so the two procedures, from coding tree to encoded order, back to coding tree are converse operations.\\

Theorem~\ref{C2} concludes our classification of countable lower 1-transitive linear orders. The companion paper \cite{me} is a major extension of this work, since it classifies countable 1-transitive trees. The branches of these trees are countable lower 1-transitive linear orders. However, two non-isomorphic trees can have branch sets where the branches are isomorphic as linear orders. In order to consider the way lower 1-transitive linear orders embed in the trees, it is necessary to consider the ramification points of the trees. These points might not be vertices of the tree, and different types of ramification points give rise to the notion of \textit{colour lower 1-transitivity}. The starting point in \cite{me} is the classification of coloured 1-transitive linear orders. In order to give a complete description of each countable 1-transitive tree, it is then necessary to consider the number and type of cones at each ramification point.


\bibliographystyle{plain}
\bibliography{refs}

\end{document}